\documentclass[12pt,reqno]{amsart}
\usepackage{amsmath,amssymb,amsthm,amsfonts}
\usepackage{graphicx}
\newtheorem{theorem}{Theorem}

\newtheorem{remark}{Remark}

\newtheorem{conj}{Conjecture}
\newtheorem{prop}{Proposition}[section]
\newtheorem{lemma}{Lemma}

\newtheorem{definition}{Definition}

\begin{document}
\date{}
\title{On a conjecture for trigonometric sums by S. Koumandos and S. Ruscheweyh}
\author{
Priyanka Sangal
}
\address{Department of Mathematics, Indian Institute of Technology, Roorkee-247 667,
Uttarkhand,  India}
\email{sangal.priyanka@gmail.com, priyadma@iitr.ac.in}

\author{
A. Swaminathan
}

\address{
Department of  Mathematics  \\
Indian Institute of Technology, Roorkee-247 667,
Uttarkhand,  India
}
\email{swamifma@iitr.ac.in, mathswami@gmail.com}
\bigskip

\begin{abstract}
S. Koumandos and S. Ruscheweyh \cite{koumandos-ruscheweyh-2007-conjecture-JAT}
posed the following conjecture: For $\rho\in(0,1]$ and $0<\mu\leq\mu^{\ast}(\rho)$,
the partial sum $s_n^{\mu}(z)=\displaystyle\sum_{k=0}^n \frac{(\mu)_k}{k!}z^k$, $0<\mu\leq1$,
$|z|<1$, satisfies
\begin{align*}
(1-z)^{\rho}s_n^{\mu}(z) \prec \left(\frac{1+z}{1-z}\right)^{\rho},
\qquad n\in \mathbb{N},
\end{align*}
where $\mu^{\ast}(\rho)$ is the unique solution of
\begin{align*}
\int_0^{(\rho+1)\pi} \sin(t-\rho\pi)t^{\mu-1}dt=0.
\end{align*}
This conjecture is already settled for $\rho=\frac{1}{2}$, $\frac{1}{4}$,
$\frac{3}{4}$ and $\rho=1$.
In this work, we validate this conjecture for an open neighbourhood of
$\rho=\frac{1}{3}$ and in a weaker form for $\rho=\frac{2}{3}$. The
particular value of the conjecture leads to several consequences related to
starlike functions.
\end{abstract}

\maketitle
\subjclass{Primary 42A05; Secondary 42A32, 33C45, 26D05}

\keywords{Keywords: Positive trigonometric sums, Inequalities, Starlike functions,
Subordination, Gegenbauer polynomials.}

\pagestyle{myheadings}\markboth{Priyanka Sangal and A. Swaminathan}
{On a conjecture for trigonometric sums by S. Koumandos and S. Ruscheweyh}
\section{Preliminaries and Main results}
Let $\mathcal{A}$ be the class of analytic functions in the unit disc
$\mathbb{D}:=\{z:|z|<1\}$. Let $s_n(f,z)$ be the $n$th partial sum of the
power series expansion of $f(z)=\sum_{k=0}^{\infty} a_k z^k$.
For $\mu>0$, let $s_n^{\mu}(z)$ be the $n$th partial sum of
$(1-z)^{-\mu}=\sum_{k=0}^{\infty} \frac{(\mu)_k}{k!}z^k$
where $(\mu)_k$ is the Pochhammer symbol defined by
$(\mu)_k=\mu(\mu+1)\cdots (\mu+k-1)$.
An important concept namely subordination plays a vital role in this work.
For two analytic functions $f,g\in\mathcal{A}$,
$f$ is subordinate to $g$, defined as $f\prec g$
by $f(z)=g(\omega(z))$ if $\omega(z)$ satisfies
$|\omega(z)|\leq |z|$, $z\in\mathbb{D}$. If $g$ is univalent in $\mathbb{D}$,
then $f(\mathbb{D})\subset g(\mathbb{D})$ and $f(0)=g(0)$ are the sufficient conditions for 
$f\prec g$.

Let $\mathcal{S}^{\ast}(\gamma)$ be the class of starlike function of order
$0\leq \gamma<1$ defined as
\begin{align*}
\mathcal{S}^{\ast}(\gamma):=\left\{f\in\mathcal{A}: f(0)=0=f'(0)-1,
\frac{zf'}{f}-\gamma \prec \frac{1+z}{1-z} \right\}
\end{align*}
for $z\in \mathbb{D}$. Clearly, the function
%
$zf_{2-2\gamma}(z):=\frac{z}{(1-z)^{2-2\gamma}}$
%
belongs to $\mathcal{S}^{\ast}(\gamma)$. For $\gamma\in[1/2,1)$,
functions of $\mathcal{S}^{\ast}(\gamma)$ are of particular interest.
In \cite{ruscheweyh-1978-kakeya-thm-SIAM} the following result was proved.
%
%
\begin{theorem}\label{thm:A}
{\rm\cite{ruscheweyh-1978-kakeya-thm-SIAM}}
Let $f\in\mathcal{S}^{\ast}(\gamma)$, $\gamma\in[1/2,1)$.
Then $s_n(f/z,z)\neq 0$, for all
$z\in\mathbb{D}$ and $n\in\mathbb{N}$.
\end{theorem}
%
Theorem \ref{thm:A} implies that the partial sums of starlike functions of order $\gamma$
where $1/2 \leq \gamma<1$ are non-vanishing in $\mathbb{D}$ and takes
value $1$ at the origin. So it is interesting to find the information about
the image domain of these polynomials. This, in particular, leads
to find the best possible range of
$\gamma$ such that $\mathrm{Re}(s_n(f/z,z))>0$ where $f\in\mathcal{S}^{\ast}(\gamma)$
and $z\in\overline{\mathbb{D}}$.
A more general statement was proved in
\cite{ruscheweyh-salinas-2000-AnnMarie} which is given as follows. For $zf\in\mathcal{S}^{\ast}(\gamma)$,
\begin{align}\label{eqn:partial-sum-conj}
\frac{s_n(f,z)}{f(z)} \prec \frac{1}{f_{2-2\gamma}}, \qquad \gamma \in[1/2,1), \quad z\in\mathbb{D},
\end{align}
which for $zf\in \mathcal{S}^{\ast}(1-\rho/2)$ where $ \rho\in(0,1]$ is equivalent to
\begin{align*}
\left|\arg s_n(f/z,z)\right|\leq \rho \pi,\qquad \mbox{for } \quad z\in \mathbb{D}.
\end{align*}
The condition \eqref{eqn:partial-sum-conj} can be replaced by a more general
condition
\begin{align}\label{eqn:maximal-number}
(1-z)^{\rho} s_n^{\mu}(z) \prec \left(\frac{1+z}{1-z}\right)^{\rho}, \quad \mbox{for all  $n\in\mathbb{N}$}
\end{align}
holds for $0<\mu \leq \mu(\rho)$ where for $\rho\in(0,1]$, $\mu(\rho)$ is defined
as maximal number. To determine the value of $\mu(\rho)$,
Koumandos and Ruscheweyh \cite{koumandos-ruscheweyh-2007-conjecture-JAT}
proposed the following conjectures.
\begin{conj}\label{conj:main}
For $\rho\in(0,1]$, the number $\mu(\rho)$ is equal to $\mu^{\ast}(\rho)$ where
$\mu^{\ast}(\rho)$ is defined to be the unique solution in $(0,1]$ of the equation
\begin{align*}
\int_{0}^{(\rho+1)\pi}\frac{\sin(t-\rho\pi)}{t^{1-\mu}}dt=0.
\end{align*}
\end{conj}
Conjecture \ref{conj:main} contains the following weaker form.
\begin{conj}\label{conj:main-weaker}
For $\rho\in(0,1]$, the following inequality
\begin{align}\label{eqn:conj-weaker-form}
\mathrm{Re}((1-z)^{2\rho-1}s_n^{\mu}(z))>0, \quad n\in\mathbb{N}
\end{align}
holds for all $z\in\mathbb{D}$ and $0<\mu\leq\mu^{\ast}(\rho)$. Moreover,
$\mu^{\ast}(\rho)$ is the largest number with this property.
\end{conj}
Generalizations of the above two conjectures are also given in
\cite{saiful-2012-stable-results-in-math} and \cite{sangal-swaminathan-stable}
by considering Ces\`aro mean of order $\delta$ and Ces\`aro mean of type $(b-1;c)$
respectively. Several interesting properties of $\mu^{\ast}(\rho)$, i.e. analytic and strictly increasing
studied in \cite{lamprecht-thesis}.
Conjecture \ref{conj:main} is verified for $\rho=1/2$, $1/4$, and $1/5$  respectively in
\cite{koumandos-ruscheweyh-2007-conjecture-JAT},
\cite{koumandos-lamprecht-2010-conjecture-JAT}
and \cite{lamprecht-thesis} while Conjecture \ref{conj:main-weaker} is verified
for $\rho=1/2$, and $\rho=3/4$ in \cite{koumandos-ruscheweyh-2007-conjecture-JAT}.
The objective of the manuscript is to establish Conjecture \ref{conj:main} for
harmonic mean of $1/2$ and $1/4$ which is $1/3$. In fact we establish Conjecture
\ref{conj:main} in the open neighbourhood of $\rho=1/3$.
Further, we also prove Conjecture \ref{conj:main-weaker} for the harmonic
mean of $1/2$ and $1$ which is $\rho=2/3$.


\begin{theorem}\label{thm:conj-2/3}
Conjecture \ref{conj:main-weaker} is true for $\rho=2/3$.
\end{theorem}

%
%
%

\begin{theorem}\label{thm:conj-1/3}
Conjecture \ref{conj:main} is true for all $\rho$ in an open neighbourhood of $\frac{1}{3}$.
\end{theorem}

The proofs of Theorem \ref{thm:conj-2/3} and Theorem \ref{thm:conj-1/3} are, respectively, given in Section
2 and Section 3. In establishing so, the method of sturm
sequence and MAPLE 12 software will be used. For more details on this method, we
refer to Kwong \cite{kwong-2014-sturm}.
Theorem \ref{thm:conj-1/3} leads to the following interesting results regarding
starlike functions. Clearly, the function
$zf_{\mu}=\frac{z}{(1-z)^{\mu}}\in\mathcal{S}^{\ast}(1-\mu/2)$. We
define a function $\phi_{\rho,\mu}:=z{}_2F_1(\rho,1;\mu;z)$ which satisfies,
\begin{align*}
\frac{z}{(1-z)^{\mu}}\ast\phi_{\rho,\mu}(z)=\frac{z}{(1-z)^{\rho}}.
\end{align*}
Here ${}_2F_1(a,b;c;z)$ is the
Gaussian hypergeometric function defined by the infinite series
$\displaystyle\sum_{k=0}^{\infty}\frac{(a)_k(b)_k}{(c)_k k!}z^k$, $z\in \mathbb{D}$.
Now for $\mu>0$ we define a class
$\mathcal{F}_{\mu}$ by,
\begin{align*}
\mathcal{F}_{\mu}:=\left\{g\in\mathcal{A}_0: \mathrm{Re}\left(\frac{zg'}{g}\right)>\frac{-\mu}{2}, z\in\mathbb{D}\right\}.
\end{align*}
Clearly, $g_{\mu}\in\mathcal{F}_{\mu}$ and $g\in\mathcal{F}_{\mu}\Leftrightarrow
zg\in\mathcal{S}^{\ast}(1-\mu/2)$.
Moreover, $g\in\mathcal{F}_{\mu}$ whenever $g\prec g_{\mu}$.
As $\mathcal{F}$ corresponds to the class $\mathcal{S}^{\ast}$, similarly,
we define class analogous to prestarlike functions by
$\mathcal{PF}_{\mu}$ as
\begin{align*}
\mathcal{PF}_{\mu}:=\left\{g\in\mathcal{A}_0: g\ast g_{\mu}\in\mathcal{F}_{\mu}\right\}.
\end{align*}
Similar type of results for prestarlike class also hold for the class
$\mathcal{PF}_{\mu}$.
A function $\tilde{g}_{\mu}\in\mathcal{A}_0$ can be defined such that
$g_{\mu}\ast\tilde{g}_{\mu}=\frac{1}{1-z}$.

\begin{theorem}\label{thm:3}
For $f\in \mathcal{S}^{\ast}(1-\mu/2)$ and $0<\mu\leq \mu^{\ast}(1/3)$, we have
\begin{align*}
\frac{s_n(f,z)}{\phi_{1/3,\mu}\ast f} \prec
\left(\frac{1+z}{1-z}\right)^{1/3},\quad
n\in \mathbb{N}.
\end{align*}
\end{theorem}
\begin{proof}
For proving the theorem we need to show that
\begin{align*}
\frac{s_n(g,z)}{\psi_{1/3,\mu}\ast g} \prec \left(\frac{1+z}{1-z}\right)^{1/3},\quad
n\in \mathbb{N},
\end{align*}
where $\psi_{\rho,\mu}(z)=\frac{1}{z}\phi_{\rho,\mu}(z)$ and
$g(z)=\frac{1}{z}f(z)\in\mathcal{F}_{\mu}$.
By writing $\frac{s_n(g,z)}{\psi_{1/3,\mu}\ast g}$ as follows and using the convolution
theory for prestarlike functions \cite[p.36]{ruscheweyh-1982-book} we obtain
\begin{align*}
\frac{s_n(g,z)}{\psi_{1/3,\mu}\ast g}
=\frac{((1-z)^{1/3}s_n^{\mu}(z)g_{1/3})\ast(g\ast \tilde{g}_{\mu})}{g_{1/3}\ast(\tilde{g_{\mu}}\ast g)}
\prec \left(\frac{1+z}{1-z}\right)^{1/3},\quad z\in\mathbb{D}
\end{align*}
where the latter inequality holds for $0<\mu\leq\mu^{\ast}(1/3)$.
\end{proof}

Theorem \ref{thm:3} also settles the particular case $\rho=1/3$ of
\cite[Conjecture 3]{koumandos-ruscheweyh-2007-conjecture-JAT}.
Validity of Theorem \ref{thm:conj-1/3} leads to the proof of Theorem \ref{thm:4}.
\begin{theorem}\label{thm:4}
For $f\in \mathcal{S}^{\ast}(1-\mu/2)$ and $0<\mu\leq \mu^{\ast}(1/3)$, we have
\begin{align}\label{eqn:coro-1/z-sn}
\frac{1}{z} s_n(f,z)\prec \left(\frac{1+z}{1-z}\right)^{2/3}.
\end{align}
\end{theorem}
\begin{proof}
Since $(1-z)^{1/3} \prec \left(\dfrac{1+z}{1-z}\right)^{1/3}$ holds.
For $f_{\mu}\in\mathcal{S}^{\ast}(1-\mu/2)$, this can be written as
\begin{align*}
\phi_{1/3,\mu}(z) \ast f_{\mu}(z) \prec \left(\frac{1+z}{1-z}\right)^{1/3},\qquad z\in\mathbb{D}.
\end{align*}
Using convolution theory of prestarlike functions for $f\in\mathcal{S}^{\ast}(1-\mu/2)$, we get
\begin{align*}
\phi_{1/3,\mu}(z) \ast f(z) \prec \left(\frac{1+z}{1-z}\right)^{1/3},\qquad z\in\mathbb{D}.
\end{align*}
Hence Theorem \ref{thm:3} concludes that
\[
\pushQED{\qed}
\frac{1}{z} s_n(f,z)=\frac{s_n(f,z)/z}{\phi_{1/3,\mu}(z) \ast f(z) }
\cdot \phi_{1/3,\mu}(z) \ast f(z) \prec \left(\frac{1+z}{1-z}\right)^{2/3}.\qedhere
\]
\end{proof}

For another interpretation of Theorem \ref{thm:4}, the following definition
of Kaplan class is of
considerable interest.
\begin{definition}[Kaplan Class]
\rm{\cite{sheilsmall-hadamard-JAnlasM-1978}}
For $\alpha\geq0$, $\beta\geq0$, $f\in\mathcal{K}(\alpha, \beta)$ if
$f(z)$ can be written as
\begin{align*}
f(z)=H(z)k(z)
\end{align*}
where $H\in\mathcal{A}$ satisfies $|\arg H(z)|\leq \frac{\pi}{2}\min(\alpha,
\beta)$ and $k(z)\in \prod_{\alpha-\beta}$ in $z\in\mathbb{D}$.
\end{definition}
The class $\prod_{\alpha}$ is defined as follows.
For real $\alpha$, a function $k(z)\in\mathcal{A}$ is in
$\prod_{\alpha}$ where $k(z)\neq0$ in $\mathbb{D}$ and satisfies
\begin{align*}
  \mathrm{Re} \frac{zk'(z)}{k(z)}
  \left\{
    \begin{array}{ll}
      <\alpha/2, & \hbox{if $\alpha>0$;} \\
      >\alpha/2, & \hbox{if $\alpha<0$;} \\
      \equiv 0, & \hbox{$\alpha=0$.}
    \end{array}
  \right.
\end{align*}
In terms of Kaplan classes $\mathcal{K}(\alpha,\beta)$
\cite[p.32]{ruscheweyh-1982-book},
\eqref{eqn:coro-1/z-sn} can be replaced by a stronger statement
\begin{align*}
\frac{1}{z} s_n(f,z) \in \mathcal{K}(1/3,2/3), \qquad n\in\mathbb{N}.
\end{align*}
\begin{remark}\label{remark:1}
Note that \eqref{eqn:coro-1/z-sn} is equivalent to Re$(s_n(f,z)/z)^{3/2}>0$
for $z\in\mathbb{D}$. It is obvious that
$z G_{\lambda}(z,x)\in\mathcal{S}^{\ast}(1-\lambda)$ where $G_{\lambda}(z,x)$
is the generating function for the Gegenbauer polynomials defined by,
\begin{align*}
G_{\lambda}(z,x):=\frac{1}{(1-2xz+z^2)^{\lambda}}
=\sum_{k=0}^{\infty}C_{k}^{\lambda}(x)z^k,
\qquad x\in[-1,1].
\end{align*}
Therefore \eqref{eqn:coro-1/z-sn} implies the following inequality holds for
all $x\in(-1,1)$ and $0<\lambda \leq \frac{1}{2}\mu^{\ast}(1/3)=0.2483\ldots$.
\begin{align*}
\left|\arg\sum_{k=0}^n C_k^{\lambda}(x)z^k\right|<\frac{\pi}{3},\qquad z\in\mathbb{D}.
\end{align*}
\end{remark}

\begin{remark}
Remark \ref{remark:1} implies that
\begin{align*}
\sum_{k=0}^n C_{k}^{\lambda}(x)z^k \neq 0\quad  \mbox{on} \quad z\in\mathbb{D}
\quad \mbox{and}\quad  0<\lambda \leq \frac{1}{2}\mu^{\ast}(1/3), \quad  x\in[-1,1].
\end{align*}
whereas from \cite[Theorem 1]{koumandos-chennai-conference-proceeding-2000}
we have
\begin{align*}
\sum_{k=0}^n C_{k}^{\lambda}(x)z^k = 0, \quad |z|\leq 1, \,\, 0<\lambda\leq 1/2.
\end{align*}
So we can conclude that
\begin{align}
\sum_{k=0}^n C_{k}^{\lambda}(x)z^k = 0 \quad \hbox{on}\quad z\in\partial
\mathbb{D}\quad \mbox{for}\quad x\in[-1,1].
\end{align}
Further by substituting $x=\cos\theta$ and using the relation between
Gegenbauer polynomials and Jacobi polynomials,
\begin{align*}
C_k^{\lambda}(x)=\frac{(2\lambda+1)_k}{k!}
\frac{P_k^{(\lambda,\lambda)}(x)}{P_k^{(\lambda,\lambda)}(1)},
\end{align*}
we obtain
\begin{align*}
\sum_{k=0}^n\frac{(2\lambda+1)_k}{k!}
\frac{P_k^{(\lambda,\lambda)}(\cos\theta)}{P_k^{(\lambda,\lambda)}(1)}\neq0,
\quad \mbox{on } z\in\mathbb{D}, \quad -\pi<\theta<\pi,
\end{align*}
for $-1/2<\lambda \leq \frac{1}{2}(\mu^{\ast}(1/3)-1)$.
\end{remark}

For any convex function $f\in\mathcal{C}$, Wilf Theorem \cite[Theorem 8.9]{duren-1983-book}
(known as Wilf Conjecture) leads us
\begin{align}
f\ast (1-z)^{1/3}s_n^{\mu}(z) \prec f\ast \left(\frac{1+z}{1-z}\right)^{1/3},\quad
0<\mu\leq \mu^{\ast}(1/3).
\end{align}

\begin{lemma}
\label{lemma:rus-sheil}
{\rm{\cite[Lemma 2.7]{ruscheweyh-sheilsmall-hadamard-product-1973}}}
Let $\varphi(z)$ be convex and $g(z)$ is starlike in $\mathbb{D}$. Then for each function
$F(z)$ analytic in $\mathbb{D}$ and satisfying $\mathrm{Re}F(z)>0$, we have
\begin{align*}
\mathrm{Re}\frac{(\varphi \ast F g)(z)}{(\varphi\ast g)(z)}>0 , \quad z\in\mathbb{D}.
\end{align*}
\end{lemma}

\begin{remark}
For any convex function $\varphi (z)$ the following holds.
\begin{align*}
  \left|\arg\left(\frac{\varphi\ast zs_n^{\mu}(z)}{\varphi\ast z(1-z)^{-1/3}}\right) \right|<\frac{\pi}{6},
  \qquad 0<\mu \leq \mu^{\ast}(1/3).
\end{align*}
\end{remark}

Since $\frac{1}{(1-z)^{1/3}}$ is the derivative of a convex function say $\varphi'$.
By writing $\dfrac{\varphi\ast zs_n^{\mu}(z)}{\varphi\ast z(1-z)^{-1/3}}$ as
\begin{align*}
\frac{\varphi\ast zs_n^{\mu}(z)}{\varphi\ast z(1-z)^{-1/3}}
=\frac{\varphi\ast \frac{z}{(1-z)^{1/3}}(1-z)^{1/3}s_n^{\mu}(z)}{\varphi\ast z(1-z)^{-1/3}}
\end{align*}
which using Lemma \ref{lemma:rus-sheil} and Theorem \ref{thm:conj-1/3} justifies the remark.

Since both the conjectures hold good for the harmonic mean of $1/2$ \& $1/4$
and $1/2$ \& $1$, affirmative solution to the following conjecture will settle 
both these conjectures.
\begin{conj}
If Conjecture \ref{conj:main} and Conjecture \ref{conj:main-weaker} hold for
$0<\rho_1, \rho_2\leq 1$. Then both conjectures also hold for harmonic mean of $\rho_1$
and $\rho_2$ which is $\dfrac{2\rho_1 \rho_2}{\rho_1+\rho_2}$.
\end{conj}

\section{Proof of Theorem \ref{thm:conj-2/3}}
Conjecture \ref{conj:main} and Conjecture \ref{conj:main-weaker} are consequences of
the positivity of trigonometric sums having
coefficients of type $\frac{(\mu)_k}{k!}$. To prove the positivity of such
trigonometric sums, we follow the procedure
given by Koumandos \cite{koumandos-2007-ext-viet-ramanujan}.
In this method, we estimate the sequence
$\left\{\frac{(\mu)_k}{k!}\right\}$ in terms of the limiting sequence
$k^{\mu-1}$. For $\mu\in(0,1)$, we consider $\Delta_k$ as
\begin{align*}
\Delta_k:=\frac{k^{\mu-1}}{\Gamma(\mu)}-\frac{(\mu)_k}{k!}, \qquad \mbox{for}
\quad k\in\mathbb{N}.
\end{align*}
Now from \cite[eq.3.8]{koumandos-ruscheweyh-2007-conjecture-JAT},
$\displaystyle\sum_{k=0}^n\frac{(\mu)_k}{k!}e^{2ik\theta}$ can be written as
\begin{align}\label{eqn:trig-sum-in-delta}
\begin{split}
\sum_{k=0}^n\frac{(\mu)_k}{k!}e^{2ik\theta}
&=\sum_{k=0}^{\infty}\frac{(\mu)_k}{k!}e^{2ik\theta}
-\frac{\theta}{\sin\theta}\frac{e^{i\mu\pi/2}}{(2\theta)^{\mu}}
+\frac{1}{\Gamma(\mu)}\frac{\theta}{\sin\theta}\frac{1}{(2\theta)^{\mu}}
\int_0^{(2n+1)\theta}\frac{e^{it}}{t^{1-\mu}}\\
&\quad \quad-\frac{1}{\Gamma(\mu)}\frac{\theta}{\sin\theta}
\sum_{k=n+1}^{\infty}\{\mathfrak{A}_k(\theta)+\mathfrak{B}_k(\theta)\}
+\sum_{k=n+1}^{\infty}\Delta_k e^{2ik\theta},
\end{split}
\end{align}
where
\begin{align*}
\mathfrak{A}_k(\theta)&:=\int_0^{1/2}\int_0^t\left\{\frac{1-\mu}{(k+s)^{2-\mu}}-
\frac{1-\mu}{(k-t+s)^{2-\mu}}\right\}ds\, e^{2i\theta(k-t)}dt \quad \hbox{and }\\
\mathfrak{B}_k(\theta)&:=\int_0^{1/2} 2i\sin 2\theta t
\int_0^t \frac{1-\mu}{(k+s)^{2-\mu}}ds \, e^{2ik\theta} dt.
\end{align*}
We have to estimate the terms in the right hand side of
\eqref{eqn:trig-sum-in-delta}.
We use the following estimates of
$\mathfrak{A}_k(\theta)$ and $\mathfrak{B}_k(\theta)$ given in 
\cite[Lemma 1]{koumandos-ruscheweyh-2006-gegenbauer-Const.Approx}:
\begin{align}\label{eqn:bounds-ak-bk}
\left|\sum_{k=n+1}^{\infty}\mathfrak{A}_k(\theta)\right|< \frac{1-\mu}{8}n^{\mu-2}
\quad \mbox{and} \quad
\left|\sum_{k=n+1}^{\infty}\mathfrak{B}_k(\theta)\right|<
\frac{\theta}{\sin\theta}\frac{1-\mu}{6}n^{\mu-2},
\end{align}
for $\theta\in\mathbb{R}$. Further, using the completely monotonicity of
$x-\frac{\Gamma(x+\mu)}{\Gamma(x+1)}x^{2-\mu}$, the following estimate was obtained in
\cite[eqn.11]{koumandos-lamprecht-2010-conjecture-JAT}.
\begin{align}\label{eqn:bounds-infinite-sum-delta-k}
\left|\sum_{k=n+1}^{\infty}\Delta_k e^{2ik\theta} \right|
\leq \frac{\mu(1-\mu)}{2\sin a} \frac{1}{\Gamma(\mu)}\frac{1}{(n+1)^{2-\mu}},
\end{align}
for $0<a<\theta<\pi/2$, $n\in\mathbb{N}$ and $1/3 \leq \mu <1$.
Now we are ready to give the proof of Theorem \ref{thm:conj-2/3}.

\subsection{Proof of Theorem \ref{thm:conj-2/3}}
For $\rho=2/3$, \eqref{eqn:conj-weaker-form} reduces to
$\mathrm{Re}((1-z)^{1/3}s_n^{\mu}(z))>0$ which further equivalent to
\begin{align}\label{eqn:U_n}
\mathfrak{U}_n(\phi):=\sum_{k=0}^n d_k \cos\left[\left(2k+\frac{1}{3}\right)\phi
-\frac{\pi}{6}\right]>0 \quad \mbox{for} \quad 0<\phi<\pi.
\end{align}
where $d_k=\frac{(\mu)_k}{k!}$, $k=0,1,\ldots, n$. To prove the theorem it is
equivalent to establish the existence of \eqref{eqn:U_n}. A numerical evaluation yields that
\begin{align*}
\mu^{\ast}(2/3)=0.8468555683\ldots.
\end{align*}
Using summation by parts, we observe that this particular case
need to be established only for $\mu=\mu^{\ast}(2/3)=0.8468555683\ldots$.
Since $\mathfrak{U}_n(\pi-\phi)=\mathfrak{U}_n(\phi)$, it is sufficient
to prove $\mathfrak{U}_n(\phi)>0$ for $0<\phi\leq\pi/2$.

\textbf{For $n=1$:}
\newline
Simple calculation gives
\begin{align*}
\mathfrak{U}_1(\phi)=(1-d_1)\sin\left(\frac{\phi}{3}+\frac{\pi}{3}\right)
+ 2d_1\sin\left(\frac{4\phi}{3}+\frac{\pi}{3}\right)\cos \phi,
\end{align*}
which is clearly positive for all $0<\phi\leq\pi/2$.
For $n\geq 2$, the proof is divided into several cases.

\textbf{The case $0<\phi\leq \dfrac{\pi}{9}$ and $\dfrac{\pi}{5}\leq \phi\leq  \dfrac{\pi}{2}$ for $n\geq2$:}
\newline
For $n\geq 2$,
\begin{align*}
2\sin\phi \mathfrak{U}_n(\phi)
=\sum_{k=0}^n d_k \left[\sin\left(\left(2k+\frac{4}{3}\right)\phi-\frac{\pi}{6}\right)
- \sin\left(\left(2k-\frac{2}{3}\right)\phi-\frac{\pi}{6}\right) \right],
\end{align*}
which using summation by parts and simple trigonometric identities yield that
\begin{align*}
&=\sum_{k=0}^{n-1}(d_k-d_{k+1})\left[ \sin \left(\frac{2\phi}{3}+\frac{\pi}{3}\right)
+ \sin \left(\left(2k+\frac{4}{3}\right)\phi-\frac{\pi}{6}\right) \right]\\
& \qquad +d_n \left[ \sin \left(\frac{2\phi}{3}+\frac{\pi}{3}\right)
+ \sin \left(\left(2n+\frac{4}{3}\right)\phi-\frac{\pi}{6}\right) \right]\\
&\geq (1-d_1)\left[\sin\left(\frac{2\phi}{3}+\frac{\pi}{6}\right)
+\sin \left(\frac{4\phi}{3}-\frac{\pi}{6}\right)\right]
+(d_1-d_2)\\
&\qquad\left[\sin\left(\frac{2\phi}{3}+\frac{\pi}{6}\right)
+\sin \left(\frac{10\phi}{3}-\frac{\pi}{6}\right)\right]+d_2 \left[-1+\sin \left(\frac{2\phi}{3}+\frac{\pi}{6}\right)\right]\\
&=-d_2+\cos\left(\frac{2\phi}{3}-\frac{\pi}{3}\right)+(1-d_1)\cos\left(\frac{4\phi}{3}-\frac{2\pi}{3}\right)
+(d_2-d_1)\cos\left(\frac{10 \phi}{3}-\frac{5\pi}{3}\right).
\end{align*}
Substituting $\dfrac{2\phi-\pi}{3}=t$, then
\begin{align*}
2\sin\phi \mathfrak{U}_n(\phi)&\geq -d_2+\cos t+ (1-d_1)\cos 2t+(d_2-d_1)\cos 5t\\
&=-\frac{\mu(\mu+1)}{2}+\cos t+ (1-\mu)\cos 2t+\frac{\mu(\mu-1)}{2} \cos 5t=:P(t).
\end{align*}
Using the method of sturm sequence, we get that $P(t)$ has no roots in
$(-\frac{\pi}{3},-\frac{7\pi}{27}]$ and $[-\frac{\pi}{5},0]$.
It can be directly verified that $P(0)>0$ and $P(-\pi/3)>0$.
Hence for $n\geq 2$,
\begin{align*}
\mathfrak{U}_n(\phi)>0 \quad \mbox{for} \quad 0<\phi\leq \pi/9 \quad
\mbox{and}\quad  \pi/5 \leq \phi \leq\pi/2.
 \end{align*}

\textbf{The case $\dfrac{\pi}{9}<\phi<\dfrac{\pi}{5}$:}
\newline
We prove this case for $n=2$, $3$ and $n\geq 4$. For $n=2$, $3$ instead of
$\frac{\pi}{9}<\phi < \frac{\pi}{5}$, we establish for the larger range 
$\phi \in (0,\pi/2]$.

\textbf{The case $n=2$, $3$:}
\newline
This can be verified again with sturm sequence. Using simple calculations,
\begin{align*}
\mathfrak{U}_2(\phi)&=\sin t+\mu \sin 7t+ \frac{\mu(\mu+1)}{2} \sin 13t=:Q(t)\quad \mbox{and}\\
\mathfrak{U}_3(\phi)&=\sin t+\mu \sin 7t+ \frac{\mu(\mu+1)}{2} \sin 13t+\frac{\mu(\mu+1)(\mu+2)}{6}\sin 19t=:R(t),
\end{align*}
where $t=\frac{\phi+\pi}{3}$. It can be verified that $Q(t)$ and $R(t)$
have no zeros in $(\frac{\pi}{3},\frac{\pi}{2}]$ and $Q(\pi/2)>0$,
$R(\pi/2)>0$. Hence it implies that $\mathfrak{U}_2(\phi)>0$
and $\mathfrak{U}_3(\phi)>0$ in $\phi\in(0,\pi/2]$.

\textbf{The case $\dfrac{\pi}{9}<\phi<\dfrac{\pi}{5}$ for $n\geq4$:}
\newline
For this case, we use the representation of $\mathfrak{U}_n(\phi)$ as
\begin{align*}
\mathfrak{U}_n(\phi)=
\mathrm{Re}\left(e^{i\left(\frac{\phi}{3}-\frac{\pi}{6}\right)}\sum_{k=0}^n
d_k e^{2ik\phi}\right)
\end{align*}
and the following propositions are used.
\begin{prop}\label{prop:F-phi}
Let
%
$F(\phi):=\displaystyle\sum_{k=0}^{\infty} d_k e^{2ik\phi}-\frac{\phi}{\sin\phi}\frac{e^{i\mu\frac{\pi}{2}}}{(2\phi)^{\mu}}$.
%
%
Then for $0<\phi\leq\frac{\pi}{5}$, we have
\begin{align*}
\frac{2^{\mu}}{\phi^{1-\mu}}\mathrm{Re} \left\{F(\phi)e^{i\left(\frac{\phi}{3}-\frac{\pi}{6}\right)}\right\}
\geq \mu \cos\left(\frac{2\pi}{3}-\frac{\mu\pi}{2}\right)-\wedge\left(\frac{\pi}{5}\right),
\end{align*}
where $\wedge(\phi):=\dfrac{1}{\sin \phi}\left(1-\left(\frac{\sin \phi}{\phi}\right)^{1-\mu}\right)$.
\end{prop}

\begin{proof}
The infinite trigonometric sum can be written as
$\displaystyle\sum_{k=0}^{\infty} d_k e^{2ik\phi}=\frac{e^{i\mu\left(\frac{\pi}{2}-\phi\right)}}{(2\sin \phi)^{\mu}}$.
Thus $F(\phi)$ becomes,
\begin{align*}
F(\phi)=\frac{\phi^{1-\mu}}{2^{\mu}}\frac{e^{i\mu\pi/2}}{\sin \phi}\left[(e^{-i\mu\phi}-1)
-\left(1-\left(\frac{\sin \phi}{\phi}\right)^{1-\mu}\right)e^{-i\mu\phi}\right].
\end{align*}
Therefore
\begin{align*}
\frac{2^{\mu}}{\phi^{1-\mu}}\mathrm{Re}\left[F(\phi)e^{i\left(\frac{\phi}{3}-\frac{\pi}{6}\right)}\right]
&= \sin\left(\frac{\mu(\pi-\phi)}{2}+\frac{\phi}{3}
-\frac{\pi}{6}\right)\frac{1}{\cos\frac{\phi}{2}}\frac{\sin\frac{\mu\phi}{2}}{\sin \frac{\phi}{2}}\\
&\qquad -\frac{1}{\sin\phi}\left[1-\left(\frac{\sin \phi}{\phi}\right)^{1-\mu}\right]
\sin\left[\left(\mu-\frac{1}{3}\right)\phi-\frac{\mu\pi}{2}+\frac{2\pi}{3}\right]\\
&\geq \mu \cos\left(\frac{2\pi}{3}-\frac{\mu\pi}{2}\right)-\dfrac{1}{\sin \phi}\left(1-\left(\frac{\sin \phi}{\phi}\right)^{1-\mu}\right).
\end{align*}
Since the function $\wedge(\phi):=\dfrac{1}{\sin \phi}\left(1-\left(\frac{\sin \phi}{\phi}\right)^{1-\mu}\right)$
is positive and strictly increasing
on $(0,\frac{\pi}{5})$, we obtain
\begin{align*}
\frac{2^{\mu}}{\phi^{1-\mu}}\mathrm{Re} \left\{F(\phi)e^{i\left(\frac{\phi}{3}-\frac{\pi}{6}\right)}\right\}
\geq \mu \cos\left(\frac{2\pi}{3}-\frac{\mu\pi}{2}\right)-\wedge\left(\frac{\pi}{5}\right),
\end{align*}
which completes the proof of Proposition \ref{prop:F-phi}.
\end{proof}

\begin{prop}\label{prop:kappa-n-phi}
Let
\begin{align*}
\chi_n(\phi):=\frac{1}{\sin(\phi)}\mathrm{Re} \left(
e^{i\left(\frac{\phi}{3}-\frac{\pi}{6}\right)}
\int_0^{(2n+1)\phi} \frac{e^{it}}{t^{1-\mu}} dt\right).
\end{align*}
For $\frac{\pi}{2n+1} \leq \phi \leq \frac{\pi}{5}$ and $n\geq 4$ we have
\begin{align*}
\chi_n(\phi)>\frac{1}{\sin \frac{\pi}{5}} \int_0^{\frac{8}{5}\pi}
\frac{\cos\left(t-\frac{\pi}{10}\right)}{t^{1-\mu}}dt=-0.3212698190821\ldots.
\end{align*}
\end{prop}

\begin{proof}
Using simple trigonometric identities we obtain,
\begin{align*}
\chi_n(\phi)=\frac{2\sin\left(\frac{\phi}{3}+\frac{\pi}{6}\right)}{\sin \phi}
\int_0^{(2n+1)\phi}\frac{\cos(t-\pi/6)}{t^{1-\mu}}dt
-\frac{2\sin\phi/3}{\sin \phi}
\int_0^{(2n+1)\phi}\frac{\cos(t-\pi/3)}{t^{1-\mu}}dt.
\end{align*}
From the definition of $\mu$, we observe that
\begin{align*}
\int_0^x \frac{\cos(t-\pi/6)}{t^{1-\mu}}dt \geq 0
\quad \mbox{and} \quad
\int_0^x \frac{\cos(t-\pi/3)}{t^{1-\mu}}dt \geq 0 \quad \mbox{for all}
\quad x>0.
\end{align*}
Further, $p(\phi):=\frac{\sin\left(\frac{\phi}{3}+\frac{\pi}{6}\right)}{\sin \phi}$
and $q(\phi):=\frac{\sin\phi}{\sin \phi/3}$
are positive and decreasing in $(0,\frac{\pi}{2})$.
Since $\phi\leq \frac{\pi}{5}$, clearly $p(\phi)\geq p(\frac{\pi}{5})$ and
$q(\phi)\geq q(\pi/5)$. Therefore,
\begin{align*}
\chi_n(\phi)\geq \frac{1}{\sin\pi/5}\int_0^{(2n+1)\phi}\frac{\cos(t-\pi/10)}{t^{1-\mu}} dt
\geq \frac{1}{\sin\pi/5}\int_0^{\frac{8}{5}\pi}\frac{\cos(t-\pi/10)}{t^{1-\mu}} dt,
\end{align*}
where the latter inequality follows by minimizing the expression on the right-hand side
over $(2n+1)\phi\geq \pi$. This completes the proof of Proposition \ref{prop:kappa-n-phi}.
\end{proof}

Now returning back to the proof of the theorem for the case $\frac{\pi}{9}<\phi<\frac{\pi}{5}$
and $n\geq 4$, let
\begin{align*}
\sigma_n(\phi):=2^{\mu}\phi^{\mu-1}\frac{\phi}{\sin \phi}\sum_{k=n+1}^{\infty}\mathfrak{A}_k(\phi)
\quad \hbox{and}\quad
\tau_n(\phi):=2^{\mu}\phi^{\mu-1}\frac{\phi}{\sin \phi}\sum_{k=n+1}^{\infty}\mathfrak{B}_k(\phi).
\end{align*}
Then,
\newline
$\mathrm{Re}\left(\sigma_n(\phi)e^{i\left(\frac{\phi}{3}-\frac{\pi}{6}\right)}\right)$
\begin{align}\label{eqn:sigma-n}
\leq 2^{\mu}\phi^{\mu-1}\frac{\phi}{\sin \phi}\frac{1-\mu}{8}\frac{1}{n^{2-\mu}}
\leq \frac{1-\mu}{4}\frac{\pi/5}{\sin \pi/5} \left(\frac{9}{2n\pi}\right)^{1-\mu}\frac{1}{n}
<\frac{1-\mu}{80}  \frac{\pi}{\sin \pi/5}
\end{align}
and
\newline
$\mathrm{Re}\left(\tau_n(\phi)e^{i\left(\frac{\phi}{3}-\frac{\pi}{6}\right)}\right)$
\begin{align}\label{eqn:tau-n}
\leq 2^{\mu}\phi^{\mu-1}\left(\frac{\phi}{\sin \phi}\right)^2\frac{1-\mu}{6}\frac{1}{n^{2-\mu}}
\leq \frac{1-\mu}{3}\left(\frac{\pi/5}{\sin \pi/5}\right)^2 \left(\frac{9}{2n\pi}\right)^{1-\mu}\frac{1}{n}
<\frac{1-\mu}{300}  \left(\frac{\pi}{\sin \pi/5}\right)^2
\end{align}
Using \eqref{eqn:bounds-infinite-sum-delta-k} we find that for $\frac{\pi}{9} \leq \phi \leq \frac{\pi}{5}$,
\begin{align}\label{eqn:delta-k-infty-bound}
2^{\mu}\phi^{\mu-1}\Gamma(\mu)\mathrm{Re}
\left(e^{i\left(\frac{\phi}{3}-\frac{\pi}{6}\right)}\sum_{k=n+1}^{\infty}
\Delta_ke^{2ik\phi}\right)
>-\frac{\mu(1-\mu)}{((2n+2)\phi)^{1-\mu}}>-\frac{\mu(1-\mu)}{\pi^{1-\mu}}.
\end{align}
Finally using Propositions \ref{prop:F-phi} - \ref{prop:kappa-n-phi} and
equations \eqref{eqn:sigma-n} - \eqref{eqn:delta-k-infty-bound}
we conclude that
\newline
$2^{\mu}\phi^{\mu-1}\Gamma(\mu)\mathfrak{U}_n(\phi)$
\begin{align*}
&>\Gamma(\mu)\left(\mu \cos\left(\frac{2\pi}{3}-\frac{\mu\pi}{2}\right)-\wedge\left(\frac{\pi}{5}\right)\right)
\quad +\frac{1}{\sin\frac{\pi}{5}}\int_0^{\frac{8\pi}{5}}\frac{\cos\left(t-\frac{\pi}{10}\right)}{t^{1-\mu}}dt
-\frac{1-\mu}{80}  \frac{\pi}{\sin \pi/5}\\
&\quad -\frac{1-\mu}{300}  \left(\frac{\pi}{\sin \pi/5}\right)^2
-\frac{\mu(1-\mu)}{\pi^{1-\mu}}>0.207809\ldots.
\end{align*}
This means $2^{\mu}\phi^{\mu-1}\Gamma(\mu)\mathfrak{U}_n(\phi)>0$ which establishes the
case $n\geq 4$ and $\frac{\pi}{9}<\phi <\frac{\pi}{5}$. Combining all these cases completes the
proof of the theorem.

\section{Proof of Theorem \ref{thm:conj-1/3}}
To prove Conjecture \ref{conj:main} for all $\rho$ in neighbourhood of
$1/3$ the following modification of $\eqref{eqn:trig-sum-in-delta}$
obtained in \cite[Lemma 1]{koumandos-lamprecht-2010-conjecture-JAT} is
required. Note that this modification plays an important role 
in proving Conjecture \ref{conj:main} for $\rho=1/4$ 
\cite{koumandos-lamprecht-2010-conjecture-JAT} and we follow the same as well. 
\begin{lemma}\cite[Lemma 1]{koumandos-lamprecht-2010-conjecture-JAT}
\label{lemma:kappa-lemma}
Let $\eta(\theta)$ be a real integrable function of $\theta\in\mathbb{R}$,
$1/3\leq \mu<1$ and $0<a<b\leq \pi/2$. Then for $g(\theta)=\sin\theta$ or
$\cos\theta$, and $\theta\in[a,b]$ we have,
\newline
$\dfrac{2^{\mu}}{\theta^{1-\mu}}\Gamma(\mu)
\displaystyle \sum_{k=0}^n \frac{(\mu)_k}{k!} g(2k\theta+\eta(\theta))
$
\begin{align*}
> \kappa_n(\theta)- X_n-Y_n-Z_n
+\Gamma(\mu) \left( 2q(\theta)\frac{\sin\mu\theta/2}{\sin\theta}-r(\theta) \wedge(\theta)\right),
\end{align*}
where
\begin{align*}
\kappa_n(\theta)&:=\frac{1}{\sin\theta}\int_0^{(2n+1)\theta}g(t+\eta(\theta))t^{\mu-1} dt, &\quad
X_n&:=\frac{b}{\sin b} \frac{1-\mu}{4n}(2an)^{\mu-1},\\
 Y_n&:=\frac{b^2}{\sin^2 b} \frac{1-\mu}{3n}(2an)^{\mu-1},
& \quad Z_n&:= \pi \mu (1-\mu)(2a(n+1))^{\mu-2},\\
 q(\theta)&:=g\left(\frac{\mu(\pi-\theta)-\pi}{2}+\eta(\theta)\right), &\quad
 r(\theta)&:=g\left(\frac{\mu(\pi-\theta)}{2}+\eta(\theta)\right)
\end{align*}
and
\begin{align*}
\wedge(\theta):=\frac{1}{\sin\theta}\left(1-\left(\frac{\sin\theta}{\theta}\right)^{1-\mu}\right).
\end{align*}
 The function $\wedge(\theta)$ is positive and increasing on $(0,\pi)$.
\end{lemma}

It follows from \cite[Lemma 2.9]{lamprecht-thesis} that proving Conjecture \ref{conj:main}
is equivalent to proving non-negativity of the trigonometric sum $\varsigma(\rho,\mu,\theta)$
for $0<\mu\leq \mu^{\ast}(\rho)$ where,
\begin{align}\label{eqn:conj-1/3-varsigma}
\varsigma_n(\rho,\mu,\theta):=\sum_{k=0}^n \frac{(\mu)_k}{k!} \sin[(2k+\rho)\theta],
\quad \theta\in[0,\pi], \quad n\in\mathbb{N}.
\end{align}
\subsection{Proof of Theorem \ref{thm:conj-1/3}}

We write $\nu_0:=\mu^{\ast}(\frac{1}{3})=0.4966913651\ldots<\frac{1}{2}$,
$\nu:=\mu^{\ast}(\rho)$, $\varsigma_n(\rho,\theta):=\varsigma_n(\rho,\mu^{\ast}(\rho),\theta)$
and $\varsigma_n(\theta):=\varsigma_n(\frac{1}{3},\theta)$ for $\rho\in(0,1)$ and $\theta\in(0,\pi]$.

At first we prove Theorem \ref{thm:conj-1/3} for $n=1,2$.
Since $\nu_0=0.49\ldots<\frac{1}{2}$, the sequence
$\left\{\frac{(\nu_0)_k}{(1/2)_k}\right\}_{k\in\mathbb{N}_0}$ is the decreasing sequence.
Therefore using summation by parts, it is enough to verify that
\begin{align*}
\omega_n(\theta):= \sum_{k=0}^n \frac{(1/2)_k}{k!}
\sin \left[\left(2k+\frac{1}{3}\right)\theta\right]>0 \quad \mbox{in}
  \quad \theta\in (0,\pi].
\end{align*}
It is easy to write $\omega_n(\theta)=\sin\theta/3 \,q_n(\cos^2 \theta/3)$ for $n=1,2$, where
\begin{align*}
q_1(x)&:=\frac{1}{2}+12 x-40x^2+32 x^3\quad \mbox{and }\\
q_2(x)&:=\frac{7}{8}-\frac{39}{2} x+380 x^2-1984 x^3+4320 x^4- 4224 x^5 +1536 x^6.
\end{align*}
Using the method of sturm sequence, we observe that $q_n(x)$ does not have any zero in
$(0,1)$ and $q_n(0)>0$ for $n=1,2$. Therefore \eqref{eqn:conj-1/3-varsigma} is
positive for $\mu=\nu_0$ and $n=1,2$. For $n\geq3$, first we observe that,
\begin{align*}
\varsigma_n(\rho,\pi-\theta)=
\sum_{k=0}^n \frac{(\nu)_k}{k!}\cos\left[\left(2k+\rho\right)\theta-\left(\rho-\frac{1}{2}\right)\pi\right]
:=\ell_n(\rho,\theta)
\end{align*}
It is obvious that $\varsigma_n(\rho,\theta)>0$ for $\theta\in(0,\pi]$ if and only if
$\ell_n(\rho,\pi-\theta)>0$. For $n\geq 3$, we split the interval $(0,\pi]$
into several subintervals.

\textbf{The case $\theta\in(0,\frac{\pi}{n+1}] \cup [\pi-\frac{\rho\pi}{n+\rho},\pi]$ and
$n\geq 3$:}
\newline
Using a simple trigonometric identity, we write
\begin{align*}
e^{im}\sum_{k=0}^n e^{ik\theta}= e^{i(m+n\theta/2)}\frac{\sin\frac{n+1}{2}\theta}{\sin \frac{\theta}{2}},
\quad n\in\mathbb{N},\quad \theta\in(0,\pi],
\end{align*}
where $m\in\mathbb{R}$ can be a function of $\theta$. Since
the sequence $\{\frac{(\nu)_k}{k!}\}_{k\in\mathbb{N}_0}$ is decreasing, again using
summation by parts and the above trigonometric identity yields that for all $\rho$
in neighbourhood of $1/3$, $\varsigma_n(\rho,\theta)>0$ for
$\theta\in(0,\frac{\pi}{n+1}]$ and $\ell_n(\rho,\theta)>0$ for $\theta\in[0,\frac{\rho\pi}{n+\rho}]$.

\textbf{The case $\theta\in [\frac{\pi}{n+1}, \frac{\pi}{3}]$ and $n\geq 3$:}
\newline
Here we apply Lemma \ref{lemma:kappa-lemma} with $g(\theta)=\sin\theta$
and $\eta(\theta)=\rho\theta$ in the interval $I=[a_n,b]$,
where $a_n=\frac{\pi}{n+1}$ and $b=\frac{\pi}{3}$. Let
$S(x):=\int_0^x \frac{\sin t}{t^{1-\nu}}dt$ and $C(x):=\int_0^x \frac{\cos t}{t^{1-\nu}}dt$.
The functions $\frac{\cos\rho\theta}{\sin\theta}$ and $-\frac{\sin\rho\theta}{\cos\theta}$
are decreasing in $(0,\pi/2)$ for $\rho\in(0,1/2)$. Therefore, we obtain
\begin{align*}
\kappa_n(\theta)
\geq \frac{\cos\rho b}{\sin b}S((2n+1)\theta)+ \rho C((2n+1)\theta)
\geq \frac{\cos\rho b}{\sin b}S(2\pi)+ \rho C(7\pi/4)=:L_1^{(1)}(\rho).
\end{align*}
where the latter inequality is obtained by minimizing the integrals
$S(x)$ and $C(x)$ over $x\geq (2n+1)\pi/(n+1)$ for $n\geq3$.
It follows from \cite[Lemma 2.13]{lamprecht-thesis}, the functions
$-q(\theta)=-\sin\left[(\nu-1)\frac{\pi}{2}+\theta(\rho-\frac{\nu}{2})\right]$
and $r(\theta)=\sin\left[\frac{\nu}{2}(\pi-2\theta)+\rho\theta\right]$ are positive and
decreasing on $[a_n,b]$. Thus we obtain,
\begin{align*}
\Gamma(\nu)\left[2q(\theta)\frac{\sin\frac{\nu\theta}{2}}{\sin\theta}-r(\theta)\wedge(\theta)\right]
\geq \Gamma(\nu)\left[2q(0)\frac{\sin\frac{\nu b}{2}}{\sin b}-r(0)\wedge(b)\right]=:L_2^{(1)}(\rho).
\end{align*}
Further for $n\geq3$ and $\theta\in I$, it is clear that the expression $X_n+Y_n+Z_n$ is smaller than
\begin{align*}
\frac{b}{\sin b}\frac{1-\nu}{12}\frac{1}{(\frac{3\pi}{2})^{1-\nu}}
+\frac{b^2}{\sin^2 b}\frac{1-\nu}{9}\frac{1}{(\frac{3\pi}{2})^{1-\nu}}
+\frac{\nu(1-\nu)\pi}{(2\pi)^{2-\nu}}=:L_3^{(1)}(\rho).
\end{align*}
Thus for $\rho$ in an open neighbourhood of $1/3$,
\begin{align*}
2^{\nu}\theta^{\nu-1}\Gamma(\nu)\varsigma_n(\rho,\theta)\geq
L^{(1)}(\rho):=L_1^{(1)}(\rho)+L_2^{(1)}(\rho)-L_3^{(1)}(\rho).
\end{align*}
Since $\nu=\mu^{\ast}(\rho)$ continuously depend on $\rho$, the function
$L^{(1)}(\rho)$ is continuous in $(0,1)$ and $L^{(1)}(\frac{1}{3})=1.00046\ldots$.
This leads to $\varsigma_n(\rho,\theta)>0$ for all $\theta\in I$, $n\geq3$ and
$\rho$ in an open neighbourhood of $\frac{1}{3}$.

\textbf{The case $\theta\in[\frac{\pi}{3},\frac{2\pi}{3}]$ for $n\geq3$:}
\newline
It can be verified that in the limiting case $n\rightarrow \infty$,
$\varsigma_n(\rho,\theta)\rightarrow \frac{\sin[(\rho-\nu)\theta+\frac{\nu\pi}{2}]}{(2\sin\theta)^{\nu}}$.
Since the sequence $\left\{\frac{(\nu)_n}{n!}\right\}$ is monotonically decreasing, we obtain
\begin{align*}
\varsigma_n(\rho,\theta)
&\geq \frac{\sin[(\rho-\nu)\theta+\frac{\nu\pi}{2}]}{(2\sin\theta)^{\nu}}
-\frac{(\nu)_{n+1}}{(n+1)!\sin\theta}\\
&\geq (2\sin2\pi/3)^{-\nu} \sin \left(\frac{\pi}{6}(4\rho-\nu)\right)- \frac{(\nu)_4}{24 \sin \pi/3}=:L^{(2)}(\rho)
\end{align*}
Again $L^{(2)}(\rho)$ is a continuous function in $\rho$ and $L(1/3)=0.0106517\ldots$.
Therefore, we obtain $\varsigma_n(\rho,\theta)>0$ for all $\rho$ in an open
neighbourhood of $1/3$ and $\theta\in[\frac{\pi}{3},\frac{2\pi}{3}]$ and $n\geq 3$.

\textbf{The case $\theta\in\left[\frac{2\pi}{3}, \pi-\frac{\rho \pi}{n+\rho}\right]$:}
\newline
For this case, instead of proving $\varsigma_n(\rho,\theta)>0$
in $\theta\in\left[\frac{2\pi}{3}, \pi-\frac{\rho \pi}{n+\rho}\right]$, we will
show the equivalent form that $\ell(\rho,\theta)>0$ in
$\left[\frac{\rho \pi}{n+\rho}, \frac{\pi}{3}\right]$. Further, we split the interval
$\left[\frac{\rho \pi}{n+\rho}, \frac{\pi}{3}\right]$ into three subintervals
$I_k=[a_k,b_k]$ where,
\begin{align*}
I_1=[a_1,b_1]:&= \left[\frac{\rho \pi}{n+\rho}, \frac{\pi}{2n+2}\right],
\quad n\geq3,\\
I_2=[a_2,b_2]:&= \left[\frac{\pi}{2n+2}, \frac{\pi}{n+2}\right],
\quad n\geq4 \quad \mbox{and}\\
I_3=[a_3,b_3]:&= \left[\frac{\pi}{n+2}, \frac{\pi}{3}\right],
\quad n\geq4.
\end{align*}
For each interval $I_k$, we estimate $\kappa_n(\theta), r(\theta),
\wedge(\theta), q(\theta), X_n, Y_n$ and $Z_n$.
It follows from \cite[Lemma 2.13]{lamprecht-thesis} that for
$0<\theta\leq b\leq \pi/2$ and $\rho\in(0,\pi/2)$, we have
$\kappa_n(\theta)\geq \mathfrak{K}(b, (2n+1)\theta)$, where
\begin{align*}
\mathfrak{K}(b,x):= \frac{1}{\sin b} \int_0^x
\frac{\cos(t+\rho b-(\rho-1/2)\pi)}{t^{1-\mu}}dt.
\end{align*}
For $\theta\in I_1$, $\kappa_n(\theta)\geq \mathfrak{K}_1((2n+1)\theta):=\mathfrak{K}(\frac{\pi}{12},(2n+1)\theta)$.
For $x\in[0,\pi]$, there is only one $x_{\mathfrak{K}_1}$ such that
$\mathfrak{K}_1(x)>0$ for $x\in [0,x_{\mathfrak{K}_1})$ and
$\mathfrak{K}_1(x)<0$ for $x\in (x_{\mathfrak{K}_1},\pi]$.
Therefore,
\begin{align*}
\mathfrak{K}_1(x)>L_1^{(31)}(\rho):= \mathfrak{K}_1(\pi).
\end{align*}
Further, for $\rho\in(0,1/2)$ and $\theta\in I_1$ the functions
$q(\theta),\wedge(\theta)$ and $-r(\theta)$ are increasing. Hence,
we obtain
\begin{align*}
\Gamma(\nu)\left(2q(\theta)\frac{\sin \nu\theta/2}{\sin\theta}-r(\theta)\wedge(\theta)\right)
\geq \Gamma(\nu)\left(\nu q(0)- r(\pi/8) \wedge(\pi/8)\right)=:L_2^{(31)}(\rho).
\end{align*}
For $n\geq3$ and $\theta\in I_1$, the expression $X_n+Y_n+Z_n$ is smaller than
$L_3^{(31)}(\rho)$, where
\begin{align*}
L_3^{(31)}(\rho):=\frac{\pi/8}{\sin \pi/8}\frac{1-\nu}{12}\frac{1}{(\frac{2\pi}{5})^{1-\nu}}
+\frac{(\pi/8)^2}{\sin^2 \pi/8}\frac{1-\nu}{9}\frac{1}{(\frac{2\pi}{5})^{1-\nu}}
+\frac{\nu(1-\nu)\pi}{(2\pi/3)^{2-\nu}}.
\end{align*}
For $\theta\in I_1$, we have
\begin{align*}
2^{\nu}\theta^{\nu-1}\Gamma(\nu)\ell_n(\rho,\theta)\geq
L_1^{(31)}(\rho)+L_2^{(31)}(\rho)-L_3^{(31)}(\rho)=:L^{(31)}(\rho).
\end{align*}
Since $L^{(31)}(\rho)$ is a continuous function in $\rho$ and
$L^{(31)}(1/3)=0.435939\ldots$. This mean $L^{(31)}(\rho)>0$ in
open neighbourhood of $1/3$. For $\theta\in I_2,I_3$ and $n\geq 4$
can be dealt in similar way. Now for $\theta\in I_2$ and $\theta\in I_3$,
$\kappa_n(\theta)\geq \mathfrak{K}_2((2n+1)\theta):=\mathfrak{K}(\frac{\pi}{6},(2n+1)\theta)$
and $\kappa_n(\theta)\geq \mathfrak{K}_3((2n+1)\theta):=\mathfrak{K}(\frac{\pi}{3},(2n+1)\theta)$
respectively. Following the reasoning similar to previous subcase, we get
$\mathfrak{K}_2(x)\geq \mathfrak{K}_2((1+5\rho/6)\pi)$ and
$\mathfrak{K}_3(x)\geq \mathfrak{K}_3(3\pi/2)$ respectively for $x\geq0$ and
$x\geq 3\pi/2$. Thus for all $\rho$ in neighbourhood of $1/3$ and $n\geq4$, we have
\begin{align*}
\kappa_n(\theta) &\geq L_1^{(32)}(\rho):= \mathfrak{K}_2((1+5\rho/6)\pi),\quad \theta\in I_2,\\
\kappa_n(\theta) &\geq L_1^{(33)}(\rho):= \mathfrak{K}_3(3\pi/2),\quad \theta\in I_3.
\end{align*}
Moreover, for $n\geq 4$ and $\theta\in I_2$ and $I_3$,
$\Gamma(\nu)\left(2q(\theta)\frac{\sin \nu\theta/2}{\sin\theta}-r(\theta)\wedge(\theta)\right)$ is
larger than $\Gamma(\nu)\left(\nu q(0)- r(\pi/6) \wedge(\pi/6)\right)=:L_2^{(32)}(\rho)$
 and
$ \Gamma(\nu)\left(\nu q(0)- r(\pi/3) \wedge(\pi/3)\right)=:L_2^{(33)}(\rho)$ respectively.
Clearly $X_n+Y_n+Z_n$ is smaller than $L_3^{(32)}(\rho)$ and $L_3^{(33)}(\rho)$
respectively for $\theta\in I_2$ and $I_3$, where $L_3^{(32)}(\rho)$ and $L_3^{(33)}(\rho)$ are given by,
\begin{align*}
L_3^{(32)}(\rho)&:= \frac{\pi/6}{\sin \pi/6}\frac{1-\nu}{16 (4\pi/5)^{1-\nu}}
+  \left(\frac{\pi/6}{\sin \pi/6}\right)^2\frac{1-\nu}{12 (4\pi/5)^{1-\nu}}
+ \frac{\nu(1-\nu)}{\pi^{1-\nu}} \quad \mbox{and}\\
L_3^{(33)}(\rho)&:= \frac{\pi/3}{\sin \pi/3}\frac{1-\nu}{16 (4\pi/3)^{1-\nu}}
+  \left(\frac{\pi/3}{\sin \pi/3}\right)^2\frac{1-\nu}{12 (4\pi/3)^{1-\nu}}
+ \frac{\nu(1-\nu)\pi}{(5\pi/3)^{2-\nu}}.
\end{align*}
Hence for $\theta\in I_2$ and $I_3$ and for all $\rho$ in an open
neighbourhood of $1/3$, we have
\begin{align*}
2^{\nu}\theta^{\nu-1}\Gamma(\nu)\ell_n(\rho,\theta)
\geq L_1^{(k)}(\rho)+L_2^{(k)}(\rho)+L_3^{(k)}(\rho)=:L^{(k)}(\rho), \quad
k=32,33.
\end{align*}
It is obvious that the functions $L^{(k)}(\rho)$ are continuous and
$L^{(32)}(1/3)=0.00620342\ldots$ and $L^{(33)}(1/3)=0.123105\ldots$. This gives 
$L^{(k)}(\rho)>0$ for $k=32,33$. The only remaining case of $\ell_3(\rho,\theta)>0$ is for
$\theta\in[\frac{\pi}{2n+2},\frac{\pi}{3}]$ and it can be verified equivalently for
$\omega_3(\rho,\theta)>0$ with $\theta\in [\frac{2\pi}{3}, \frac{7\pi}{8}]$
using method of sturm sequence. Now $\omega_3(\theta)$ can be written as
$\omega_3(\theta)=\sin\theta/3 q_3(\cos^2\theta/3)$ for
$\theta\in[\frac{2\pi}{3}, \frac{7\pi}{8}]$, where
\begin{align*}
q_3(x)&:=\frac{9}{16}+ \frac{147}{4}x - 1270 x^2 + 16496 x^3 -98640 x^4
+316096 x^5 -580864 x^6 \\
&\quad \quad +614400 x^7 -348160 x^8 + 81920 x^9.
\end{align*}
Now $q_3(x)$ does not have any zero in $[0.37059,1]$ and $q(1)>0$ yields that
$\omega_3(\theta)>0$ and this completes the proof of the theorem.

\textbf{Acknowledgement:}{ The first author is thankful to the
“Council of Scientific and Industrial Research, India”
(grant code: 09/143(0827)/2013-EMR-1) for financial support to carry out
the above research work.
}

\end{document}